\documentclass[12pt]{amsart}
\usepackage{mathrsfs}
\usepackage{amsthm}
\usepackage{amssymb}
\usepackage{amsmath}
\numberwithin{equation}{section}
\newcommand{\bea}{\begin{eqnarray}}
\newcommand{\eea}{\end{eqnarray}}
\newcommand{\be}{\begin{eqnarray*}}
\newcommand{\ee}{\end{eqnarray*}}
\newtheorem{theorem}{Theorem}[section]

\newtheorem{corollary}{Corollary}[section]

\newtheorem{proposition}{Proposition}[section]

\begin{document}
\title[Ideals of Clifford Algebras] {Ideal Structure of Clifford Algebras}
\author[Y. M. Zou]{Yi Ming Zou}
\address{Department of Mathematical Sciences, University of Wisconsin, Milwaukee, WI 53201, USA} \email{ymzou@uwm.edu}
\thanks{The original publication is available at www.springerlink.com}
\begin{abstract}
    The structures of the ideals of Clifford algebras which can be both infinite dimensional and degenerate over the real numbers are investigated. 
\end{abstract}
\maketitle
\section{Introduction}
\par
Clifford algebras have been playing an important role in describing electron spin, supersymmetry, and the fundamental representations of the orthogonal groups etc. We refer the readers to the introduction section of [7] for a discussion of the role played by Clifford algebras in quantum mechanics. Recently, there has been an increase of interest in the applications of Clifford algebras to other areas such as computational geometry and engineering. In these applications, the real Clifford algebras are usually called geometric algebras. It is shown in [7] that in general an infinite dimensional Clifford algebra over the real numbers contains all finite dimensional Clifford algebras over the real numbers as well as all finite dimensional Clifford algebras over the complex numbers. Therefore it is natural to study the structure of a real Clifford algebra under the general setting which allows both infinite dimensionality and degeneracy.
\par
The structures of the finite dimensional Clifford algebras associated to nondegenerate quadratic forms have been well understood for a long period of time [2]. These Clifford algebras are either full matrix algebras or the direct sums of two full matrix algebras over the real numbers $\mathbb{R}$, or the complex numbers $\mathbb{C}$, or the quaternions $\mathbb{H}$. The study of the idempotent structures of the finite dimensional Clifford algebras over the real numbers, including the degenerate cases, was carried out in [1] and [5]. Somewhat later, [7] and [8] investigated the structures of the infinite dimensional Clifford algebras over the real numbers. As pointed out in [7], the infinite dimensional Clifford algebras over the real number arise naturally in physics when an infinite dimensional real Hilbert space is considered, and these algebras are connected to the infinite dimensional orthogonal groups and spin groups. Although in general an infinite dimensional Clifford algebra contains the finite dimensional ones, the study of the structure of an infinite dimensional Clifford algebra relies on the knowledge of the structures of the finite dimensional ones. Our goal here is to answer some of the questions about the ideal structures of these algebras under the assumption that the Clifford algebras can be finite or infinite dimensional as well as degenerate. 
\par 
\section{Preliminaries}
We recall the definition of a general Clifford algebra [7]. Let $M, P, Z\subseteq\mathbb{Z}_{+}=\{0,1,2,\ldots\}$ and let $\Lambda =M\cup P\cup Z$. Let $q=\mid M\mid$, $p=\mid P\mid$, and $z=\mid Z\mid$. We allow $p$, $q$, and $z$ to be $\infty$. Let $\mathscr{C}=\mathscr{C}(p, q, z)$ be the Clifford algebra over the real numbers $\mathbb{R}$ generated by $\{e_{\lambda} : \lambda\in \Lambda\}$. That is, $\mathscr{C}$ is the associative algebra (with $1$) generated by the $e_{\lambda}, \lambda\in\Lambda$, such that
\bea
e_{\lambda}e_{\mu}=-e_{\mu}e_{\lambda}, \quad \lambda\ne\mu,\\
e_{\lambda}^{2}=\left\{ \begin{array}{rl} 
1, & \lambda\in P,\\
-1, & \lambda\in M,\\
0, & \lambda\in Z.
\end{array}\right.
\eea
By an ordered subset $\{i_{1}, i_{2}, \ldots \}$ of $\mathbb{Z}_{+}$ we mean that the condition $i_{1}<i_{2}<\cdots$ hold, and we shall denote an ordered subset of $\mathbb{Z}_{+}$ by $I=(i_{1},i_{2},\ldots)$. For a finite ordered subset $I=(i_{1},\ldots, i_{k})$ of $\mathbb{Z}_{+}$, let $e_{I}=e_{i_{1}}\cdots e_{i_{k}}$. If $I=\emptyset$, then $e_{I}=1$. When we write $e_{I}e_{J}e_{K}$, we always assume that $I,J,K$ are finite ordered sets.
\par
It follows from the definition that the set of elements
\bea
\{e_{I}e_{J}e_{K}: I\subseteq P, J\subseteq M, K\subseteq Z\} 
\eea
form a basis of $\mathscr{C}$.
\par 
We let $\mathscr{C}(p,q)$ be the subalgebra generated by $\{e_{\lambda} : \lambda\in M\cup P\}$, and let $\mathscr{C}(z)$ be the subalgebra generated by $\{e_{\lambda} : \lambda\in Z\}$. 
\par
If $z=0$, then $\mathscr{C}=\mathscr{C}(p,q)$, and we divide the algebras into two cases: $p+q<\infty$ and $p+q =\infty$.
\par 
When $p+q=\infty$, $\mathscr{C}$ is simple [7] and the idempotent structure is studied in [8]. In particular, it is proved in [8] that there is no primitive idempotent in $\mathscr{C}$ and there exists an infinite set of pair-wise orthogonal idempotents. 
\par
If $p+q <\infty$, then $\mathscr{C}$ is simple (a full matrix algebra) if $p-q\ne 1, 5$ modulo 8, and $\mathscr{C}$ is a direct sum of two isomorphic simple ideals (two full matrix algebras) if $p-q= 1, 5$ modulo 8 [6, pp.132-133]. Therefore, in this case, the information about the ideals and idempotents of $\mathscr{C}$ can be readily obtained by using the corresponding results of matrices. For our convenience, if $\mathscr{C}$ is not simple, we write 
\bea
\mathscr{C}=\mathscr{C}_{1}\oplus\mathscr{C}_{2},
\eea
where $\mathscr{C}_{1}$ and $\mathscr{C}_{2}$ are the two simple ideals.
\par
Recall that an ideal $I$ of an algebra (associative) $A$ is nil if for $\forall x\in A$ there exists a positive integer $n$ such that $x^{n}=0$, and the nil radical $\mathscr{R}$ of $A$ is the sum of all nil ideals.
\par
If $z>0$, then the nil radical $\mathscr{R}$ of $\mathscr{C}$ is non-zero. By (2.3) and Proposition 2.1 below, we have that  
\bea
\mathscr{C}=\mathscr{C}(p,q)\oplus\mathscr{R}
\eea
is a vector space direct sum. In this case, the structure of the idempotents of a finite dimensional $\mathscr{C}$ is discussed in [1] by using the fact that the nil radical is idempotent-lifting.
\par
The following proposition is given in [7], we provide a complete proof here. 
\begin{proposition}
The nil radical $\mathscr{R}$ is generated by $\{e_{\lambda}: \lambda\in Z\}$.
\end{proposition}
\begin{proof}
Let $N$ be the ideal generated by $\{e_{\lambda}: \lambda\in Z\}$. Then by (2.3),
\bea
\mathscr{C}=\mathscr{C}(p,q)\oplus N
\eea 
is a direct sum of vector spaces. Since the nil radical $\mathscr{R}$ is the sum of all nil ideals, the ideal $N$ is contained in $\mathscr R$. If $a\in \mathscr{R}$, write $a=a_{1}+a_{2}$ such that $a_{1}\in \mathscr{C}(p,q)$ and $a_{2}\in N$ according to (2.6), then $a_{1}\in \mathscr{R}$. If $a_{1}\ne 0$, then we can get a contradiction right away if the subalgebra $\mathscr{C}(p,q)$ is simple. In the case that $\mathscr{C}(p,q)$ is not simple, $\mathscr{C}(p,q)$ is a full matrix algebra or a direct sum of two copies of a full matrix algebra. Thus by considering the matrices (say, use elementary matrices from both sides), we can see that there is an idempotent in the ideal generated by $a_{1}$, which also leads to a contradiction. So $a_{1}=0$ and $N=\mathscr{R}$.  
\end{proof}
\par
It follows from Proposition 2.1 that the elements $e_{I}e_{J}e_{K}$, where $I\subseteq P, J\subseteq M, K\subseteq Z$ such that $K\ne\emptyset$, form a basis of $\mathscr{R}$. The subalgebra $\mathscr{C}(z)$ has a grading given by
\be
\mathscr{C}(z)=\bigoplus_{i=1}^{\infty}\mathscr{C}(z)_{i},
\ee
where $\mathscr{C}(z)_{i}$ is the linear span of all the $e_{K}$ ($K\subseteq Z$) such that $\mid K\mid=i$. This grading induces a grading on $\mathscr{R}$ given by
\be
\mathscr{R}=\bigoplus_{i=1}^{\infty}\mathscr{R}_{i},\quad \mathscr{R}_{i}=\mathscr{C}(p,q)\mathscr{C}(z)_{i}.
\ee 
\par
\section{The ideal structure of $\mathscr{C}$}
\par
In this section, we describe the ideals of $\mathscr{C}$.
\par
\begin{theorem} Let $I$ be a non-trivial ideal of $\mathscr{C}$. If $\mathscr{C}(p,q)$ is simple, then $I\subseteq\mathscr{R}$. If $\mathscr{C}(p,q)$ is not simple and $I\nsubseteq\mathscr{R}$, then $I$ is generated by $\mathscr{C}_{1}$ or $\mathscr{C}_{2}$ together with $I\cap\mathscr{R}$. In fact, 
\be
I=\mathscr{C}_{1}\oplus(I\cap\mathscr{R})\quad\text{or}\quad I=\mathscr{C}_{2}\oplus(I\cap\mathscr{R})
\ee 
as vector space direct sums.
\end{theorem}
\begin{proof}
If $\mathscr{C}(p,q)$ is simple and $I\nsubseteq\mathscr{R}$, then there exists $a=b+c\in I$ such that $0\ne b\in\mathscr{C}(p,q)$ and $c\in\mathscr{R}$. Since
\be
\mathscr{C}(p,q)b\mathscr{C}(p,q)= \mathscr{C}(p,q),
\ee
from $b+c\in I$ we get $1+x\in I$ for some $x\in\mathscr{R}$. Since $x$ is nilpotent, $1+x$ is invertible and $I=\mathscr{C}$, which is a contradiction.
\par
If $\mathscr{C}(p,q)$ is not simple and $I\nsubseteq\mathscr{R}$, then there exists $a=b+c+d\in I$ with $b\in\mathscr{C}_{1}$, $c\in\mathscr{C}_{2}$, and $d\in \mathscr{R}$, such that at least one of $b$ and $c$ is non-zero. Let $1=e_{1}+e_{2}$ be the idempotent decomposition according to the decomposition of (2.4). Then $e_{1}a=b+e_{1}d\in I$. Thus, if $b\ne 0$, then $e_{1}+d_{1}\in I$ for some $d_{1}\in\mathscr{R}$. Similarly, if $c\ne 0$, then $e_{2}+d_{2}\in I$ for some $d_{2}\in\mathscr{R}$. Thus if both $b$ and $c$ are non-zero (or for some $a$, $b\ne 0$, and for another $a$, $c\ne 0$), we would have
\be
e_{1}+d_{1}+e_{2}+d_{2}=1+d_{1}+d_{2}\in I.
\ee
Now the fact that $d_{1}+d_{2}\in\mathscr{R}$ implying that $1+d_{1}+d_{2}$ is invertible would lead to a contradiction. Hence without loss of generality, we can assume that $b\ne 0$ and $c=0$ for all $a\in I\setminus\mathscr{R}$. Then from $e_{1}+d_{1}\in I$, we have
\be
e_{1}(e_{1}+d_{1})e_{1}=e_{1}+e_{1}d_{1}e_{1}\in I.
\ee
Let $y=-e_{1}d_{1}e_{1}$. Then $y\in\mathscr{R}$, $e_{1}y=y=ye_{1}$, and $e_{1}-y\in I$. Take a positive integer $m$ such that $y^{m}=0$, then
\be
e_{1}=(e_{1}-y)(e_{1}+y+y^{2}+\cdots+y^{m-1})\in I.
\ee
This implies that $\mathscr{C}_{1}\subseteq I$. Now the second statement of the theorem follows from (2.5).
\end{proof}
\par
This theorem reduces the study of the nontrivial ideals of $\mathscr{C}$ to the study of the ones contained in the nil radical $\mathscr{R}$. We shall give a description of the nilpotent ideals of $\mathscr{C}$ towards the end of this section. 
\par
Recall that an ideal $I$ of a noncommutative ring $R$ is said to be prime if $I\ne R$ and for any ideal $A, B \subseteq R$, $AB\subseteq I$ implies $A\subseteq I$ or $B\subseteq I$. The following theorem describes the prime ideals of $\mathscr{C}$.
\begin{theorem} If the subalgebra $\mathscr{C}(p,q)$ is simple, then $\mathscr{C}$ has only one prime ideal, namely $\mathscr{R}$. If the subalgebra $\mathscr{C}(p,q)$ is not simple, then there are two prime ideals:
\be
I=\mathscr{C}_{1}\oplus \mathscr{R}\quad\text{and}\quad I=\mathscr{C}_{2}\oplus \mathscr{R}.
\ee 
\end{theorem}
\begin{proof}
Let $I$ be a prime ideal of $\mathscr{C}$. Since $e_{i}^{2}=0$ for $i\in Z$, $I\supseteq\mathscr{R}$. If $\mathscr{C}(p,q)$ is simple, then $\mathscr{R}$ is a prime ideal. If $\mathscr{C}(p,q)=\mathscr{C}_{1}\oplus\mathscr{C}_{2}$, then $\mathscr{R}$ is not prime since $\mathscr{C}_{1}\mathscr{C}_{2}=(0)$, and in this case, $\mathscr{C}_{i}\oplus\mathscr{R}$, $i=1,2$, are the two prime ideals.
\end{proof}
\par
Recall that in a ring $R$, we have the implications: maximal ideal $\Longrightarrow$ primitive ideal $\Longrightarrow$ prime ideal (see Ch. 3 of [3]). Theorem 3.2 immediately implies the following corollary:
\begin{corollary} In the algebra $\mathscr{C}$, maximal ideal $=$ primitive ideal $=$ prime ideal.
\end{corollary}
\par
By using the descriptions of the Jacobson radical of a ring on p. 196 of [4], we also obtain the following corollary:
\begin{corollary} The Jacobson radical $J(\mathscr{C})$ is equal to $\mathscr{R}$.
\end{corollary}
\par
Although the fact that there exists an infinite set of pair-wise orthogonal idempotents in $\mathscr{C}$ when $p+q=\infty$ implies that $\mathscr{C}$ is not left Noetherian (hence not left Artinian), we do have the following theorem (note that the proof also shows that if $z=\infty$, then $\mathscr{C}$ is not left Noetherian even if $p+q<\infty$).
\begin{theorem} The algebra $\mathscr{C}$ is Artinian (Noetherian) if and only if $p+q<\infty$ and $z<\infty$.
\end{theorem}
\begin{proof} 
If $z=\infty$, we can assume that $Z=\mathbb{Z}_{+}$ and construct an infinite descending chain of ideals by letting $I_{i}=(e_{0}e_{1}\cdots e_{i})$ ($i\in \mathbb{Z}_{+}$). Thus in this case, the algebra is not Artinian. It is also not Noetherian since the ascending chain of ideals $I_{i}$ generated by $\{e_{0},\ldots,e_{i}\}$ ($i\in \mathbb{Z}_{+}$) does not terminate.
\par
Suppose that $z$ is finite. There is nothing to prove if $z=0$, so we assume that $z>0$. We prove the statement that $\mathscr{C}$ is Artinian, the proof for Noetherian is similar. Since $\mathscr{C}(p,q)$ is either simple or a direct sum of two simple ideals, $\mathscr{C}/\mathscr{R}$ is Artinian. So to prove that $\mathscr{C}$ is Artinian, we only need to prove that $\mathscr{R}$ is Artinian. We use induction on $z$. If $z=1$, then $\mathscr{R}=\mathscr{C}(p,q)e_{0}$, where $e_{0}$ is the only generator of $\mathscr{C}(z)$, is a simple ideal or is a direct sum of two simple ideals
\be
\mathscr{R}=\mathscr{C}_{1}e_{0}\oplus\mathscr{C}_{2}e_{0},
\ee 
and hence is Artinian. Suppose $\mathscr{C}$ is Artinian for $z=k\ge 1$, consider the case $z=k+1$. Let $Z=\{0, 1,\ldots, k\}$. We claim that the ideal $(e_{0})$ generated by $e_{0}$ is Artinian.
\par
In fact, the ideal $\mathscr{R}_{0}=(e_{1},\ldots, e_{k})$ is Artinian by assumption, and 
\be
(e_{0})=\mathscr{R}_{0}e_{0}\oplus\mathscr{C}(p,q)e_{0},
\ee
so $(e_{0})$ is Artinian since both $\mathscr{R}_{0}e_{0}$ and $(e_{0})/\mathscr{R}_{0}e_{0}$ are. Now 
\be
\mathscr{R}/(e_{0})\simeq\mathscr{R}_{0},
\ee
implies that $\mathscr{R}$ is Artinian.
\end{proof}
\par
We now describe the nilpotent ideals of $\mathscr{C}$. For a basis element $e_{I}e_{J}e_{K}$ described in (2.3), let 
\be
N(e_{I}e_{J}e_{K})=\{e_{k}:k\in K\}.
\ee
In general, for
\be
u=\sum_{IJK} c_{IJK}e_{I}e_{J}e_{K}\in\mathscr{C},
\ee
let
\be
N(u)=\bigcup_{IJK} N(e_{I}e_{J}e_{K}).
\ee 
That is, $N(u)$ is the set of the $e_{\lambda}$ ($\lambda\in Z$) that occur in the expression of $u$ via the basis described by (2.3). For a subset $S\subseteq Z$, let $I_{S}$ be the ideal of $\mathscr{C}$ generated by $\{e_{\lambda}:\lambda\in S\}$.
\begin{theorem}
Let $I\subseteq\mathscr{R}$ be an ideal of $\mathscr{C}$. Then $I$ is nilpotent if and only if there exists a finite subset $S\subseteq Z$ such that $I\subseteq I_{S}$.
\end{theorem}
\begin{proof}
We can assume that $z=\infty$, since the proof is only needed in this case. Since $I_{S}$ is nilpotent if $S$ is finite, the condition is sufficient. Assume that $I$ is a nilpotent ideal of $\mathscr{C}$. We claim that if $I\not\subseteq I_{S}$ for any finite $S\subseteq Z$, then we can  construct a sequence of elements $(u_{1},u_{2},\ldots)$ in $I$ such that the product $u_{1}u_{2}\cdots u_{t}\ne 0$ for any $t\ge 1$, and thus obtain a contradiction. 
\par
Start with $u_{1}\ne 0$ in $I$, set $a_{1}$ to be a nonzero term in the expression of $u_{1}$ via the basis of (2.3). Since $N(u_{1})$ is finite, $I\not\subseteq (N(u_{1}))$. Thus we can choose $u_{2}\ne 0$ in $I$, such that if 
\be
u_{2}=\bigoplus_{IJK} c_{IJK}e_{I}e_{J}e_{K},
\ee 
then there exists a nonzero term 
\be
a_{2}=c_{IJK}e_{I}e_{J}e_{K}
\ee
with $N(e_{K})\cap N(u_{1})=\emptyset$. Now since $N(u_{1})\cup N(u_{2})$ is finite, we can choose $u_{3}\ne 0$ in $I$ such that there exists a term 
\be
a_{3}=d_{IJK}e_{I}e_{J}e_{K}
\ee 
in the expression of 
\be
u_{3}=\bigoplus_{IJK} d_{IJK}e_{I}e_{J}e_{K}
\ee 
with 
\be
N(u_{3})\cap(N(u_{1})\cup N(u_{2}))=\emptyset.
\ee 
Continuing this process, we get an infinite sequence of nonzero elements $(u_{1}, u_{2}, \ldots)$ in $I$. Note that if $\mathscr{C}(p,q)$ is simple, then $u_{1}u_{2}\cdots u_{t}\ne 0$ for any $t\ge 1$, and we are done.  If $\mathscr{C}(p,q)$ is not simple, a product $u_{1}u_{2}\cdots u_{t}$ can be zero since for two terms 
\be
e_{I_{i}}e_{J_{i}}e_{K_{i}}, \quad i=1,2, 
\ee
the $\mathscr{C}(p,q)$ parts $e_{I_{i}}e_{J_{i}}, i=1,2,$ may belong to different simple ideals of $\mathscr{C}(p,q)$ (see (2.4)). Thus if $\mathscr{C}(p,q)$ is not simple, we need to further choose a subsequence $(v_{1},v_{2},\ldots)$ of $(u_{1}, u_{2}, \ldots)$ such that the projections of the $\mathscr{C}(p,q)$ parts of the corresponding $a_{i}$ (see above) to one of the simple ideals of $\mathscr{C}(p,q)$ are all nonzero to obtain the required sequence. This is possible since at least one of the projections of each of the $a_{i}$ to the two simple ideals of $\mathscr{C}(p,q)$ must be nonzero.
\end{proof}
\par
Since $I_{S}$ is Noetherian for any finite subset $S\subseteq Z$ by Theorem 3.3, we have the following corollary:
\begin{corollary}
Every nilpotent ideal of $\mathscr{C}$ is finitely generated.
\end{corollary}
\par
\medskip

\end{document}